\documentclass[11pt]{amsart}

\usepackage{amssymb,xfrac}
\usepackage{enumerate}
\usepackage{tikz}
\usetikzlibrary{arrows}

\newtheorem{thm}{Theorem}[section]
\newtheorem{prop}[thm]{Proposition}
\newtheorem{lem}[thm]{Lemma}

\newtheorem{question}[thm]{Question}
\newtheorem{conj}[thm]{Conjecture}
\newtheorem{problem}[thm]{Problem}

\def\ex{{\operatorname{ex}}}
\def\exr{{\operatorname{ex}^\star}}
\def\sat{{\operatorname{sat}}}
\def\satr{{\operatorname{sat}^\star}}
\def\satrr{{\operatorname{sat}^\star_r}}

\def\satrr{{\operatorname{sat}^\star_r}}
\def\F{{\mathcal{F}}}

\theoremstyle{definition}
\newtheorem{definition}[thm]{Definition}

\theoremstyle{remark}

\newtheorem{remark*}{Remark}

\numberwithin{equation}{section}

\title{Rainbow Saturation}

\author{Neal Bushaw}
\address{Dept. of Math. \& Appl. Math., Virginia Commonwealth University}
\email{nobushaw@vcu.edu}

\author{Daniel Johnston}
\address{Dept. of Mathematics \& Statistics, Skidmore College}
\email{djohnst1@skidmore.edu}

\author{Puck Rombach}
\address{Dept. of Mathematics \& Statistics, University of Vermont}
\email{puck.rombach@uvm.edu}

\begin{document}
\begin{abstract}
    This paper explores a new direction in the well studied area of graph saturation -- we introduce the notion of rainbow saturation.  A graph $G$ is rainbow $H$-saturated if there is some proper edge coloring of $G$ which is rainbow $H$-free (that is, it has no copy of $H$ whose edges are all colored distinctly), but where the addition of any edge makes such a rainbow H-free coloring impossible.
    
    Taking the maximum number of edges in a rainbow $H$-saturated graph recovers the rainbow Tur\'an numbers whose systematic study was begun by Keevash, Mubayi, Sudakov, and Verstra\"ete.  In this paper, we initiate the study of corresponding {\emph{rainbow saturation number}} -- the minimum number of edges among all rainbow $H$-saturated graphs.  We give examples of graphs for which the rainbow saturation number is bounded away from the traditional saturation number (including all complete graphs $K_n$ for $n\geq 4$ and several bipartite graphs). It is notable that there are non-bipartite graphs for which this is the case, as this does not happen when it comes to the rainbow Tur\'an number versus the traditional Tur\'an number. We also show that saturation numbers are at most linear for a large class of graphs, providing a partial rainbow analogue of a well known theorem of K\'aszonyi and Tuza. We conclude this paper with a number of related open questions and conjectures.
    
\end{abstract}


\maketitle

\section*{Declarations}
\subsection*{Funding} None.
\subsection*{Conflicts of interest/Competing interests} None.
\subsection*{Availability of data and material} Not applicable.
\subsection*{Code availability} Not applicable.

\section{Introduction / History}


Among the oldest problems in graph theory is the forbidden subgraph problem -- how can we characterize the set of graphs forbidding some fixed graph $H$?  In the early part of the last century, this gave birth to the prototypical problem in  extremal graph theory -- how many edges can an $n$-vertex $H$-free graph have?

We say that a graph $G$ is $H$-saturated when it contains no copy of $H$, but the addition of any edge creates some $H$. The earliest known problem studied in this context was proposed by Mantel as an exercise in a publication by the Royal Dutch Mathematical Society, and answered independently by Gouwentak, Mantel, Teixeira de Mattes, Schuh and Wythoff~\cite{MantelSol} (only Wythoff's version is published): how many edges can an $n$-vertex triangle-free graph have?

Further work in this area did not proceed until the seminal work of P\'al Tur\'an in 1941 \cite{Turan}, providing a generalization of the result to $K_r$-free graphs.

In this paper, however, we go the other direction: what is the {\emph{smallest}} number of edges a (maximal) $K_n$-free graph could have?  This, too, is a question with a long and involved history. For a state of the art survey, see the recent surveys~\cite{CFFS,RG}.

Motivated by an important classical problem in additive number theory, Keevash, Mubayi, Sudakov, and Verstra\"ete~\cite{KMSV} introduced the {\emph{rainbow extremal number}}. This is in many ways a problem at the intersection of Tur\'an Theory and Ramsey Theory -- we consider (properly) edge colored graphs, and the forbidden graph is a copy of $H$ in which every edge is colored differently (we call such a copy of $H$ {\bf{rainbow}}). As the main purpose of this paper is to introduce a similar rainbow saturation number, we give a careful definition. Since it doesn't complicate the definition, we give here a definition for $k$-uniform hypergraphs and not just for the $2$-graphs which are the focus of this manuscript. In all of the definitions below, we will typically omit the subscript for the case $r=2$.

All graphs in this paper are finite, and without loops or multiple edges. Wherever possible, we use standard notation (see, e.g.,~\cite{MGT}).

\begin{definition}
A {\bf{$k$-edge-coloring}} of an $r$-uniform hypergraph $G$ is an assignment $\phi:E(G)\to\{1,2,\ldots,k\}$. We call such a coloring {\bf{proper}} if whenever two hyperedges $e$ and $f$ share at least one vertex, we have $\phi(e)\neq\phi(f)$.
\end{definition}

\begin{definition}
Given an $r$-uniform hypergraph $H$ and $n\in\mathbb{N}$, the {\bf{extremal function for H}} (denoted by $\exr_r(n,H)$) is the largest number of edges in any $n$-vertex $r$-uniform hypergraph $G$ with the following properties:
\begin{enumerate}[(a)]
\item There is a proper edge coloring of $G$ containing no rainbow copy of $H$,
\item For every $r$-tuple $e\not\in E(G)$, every proper edge coloring of $G+e$ contains a rainbow copy of $H$
\end{enumerate}
We call a hypergraph satisfying (a) and (b) {\bf{rainbow $H$-saturated}}.
\end{definition}

Building on this, it is then entirely natural to define the analogous rainbow saturation number.

\begin{definition}
Given a graph $H$ and $n\in\mathbb{N}$, we denote by $\satrr(n,H)$ the smallest number of edges in any $n$-vertex rainbow $H$-saturated graph.
\end{definition}

Later, we will need a similar definition where the forbidden subgraph is actually a family of graphs $\F=\{F_1,\ldots,F_t\}$. In this regime, a graph $G$ is rainbow $\F$-saturated when $G$ has a proper edge coloring avoiding a rainbow copy of {\emph{every}} graph in the family $\F$, but the addition of a new edge creates, under {\emph{any}} coloring, a rainbow copy of {\emph{some}} graph $F\in\F$. The rainbow saturation number for a family is then the smallest number of edges in such a graph.

It is worth noting that the phrase `rainbow saturation' has appeared in the literature, but in a somewhat different context (see, e.g.,~\cite{rain1,rain2}). However, as the version presented in this paper is analogous to the rainbow extremal number in every way, we insist on using this terminology. To the authors' knowledge, this version has not yet been explored in the literature.

\section{Lower bounds on $\satr(n,K_k)$}
In this Section, we give a lower bound on saturation numbers for complete graphs. For comparison, we first state the corresponding result in the non-rainbow case, due to Erd\H{o}s, Hajnal, and Moon~\cite{EHM}.

In the statement, we use $G+H$ to denote the graph join of $G$ and $H$, the graph consisting of disjoint copies of $G$ and $H$ along with every possible edge in between; i.e., $V(G+H)=V(G)\cup V(H)$ and $E(G+H)=E(G)\cup E(H)\cup\{xy:x\in V(G),y\in V(H)\}$. This will show up repeatedly in the remainder of this document.

\begin{thm}[Erd\H{o}s-Hajnal-Moon 1964]
For $2\le r\le n$, we have $$\sat(n,K_r)=(r-2)(n-r+2)+\binom{r-2}{2}.$$
Further, the unique extremal graph is $K_{r-2}+E_{n-r+2}$, where $E_{n-r+2}$ is the empty graph on $n-r+2$ vertices.
\end{thm}

\begin{thm}\label{thmkr}
For $ 4\leq r \leq n$, we have $\satr(n,K_r)\geq \sat(n,K_r)+n-2r+3$.
\end{thm}

\begin{proof}
First, notice that if $G$ is rainbow $K_r$-saturated then any two nonadjacent vertices $x,y$ have at least $r-2$ neighbors in common, since adding $xy$ creates a $K_r$. Suppose that there are two nonadjacent vertices $x,y$ that have exactly $r-2$ neighbors in common. This neighborhood must be a $K_{r-2}$. Copying a non-incident color from this $K_{r-2}$ (in a rainbow $K_r$-free coloring of $G$) we color $xy$ to create a rainbow $K_r$-free coloring of $G+xy$, contradicting that $G$ is rainbow $K_r$-saturated. Therefore, any two nonadjacent vertices in $G$ have at least $r-1$ neighbors in common.\\

Let $v$ be a vertex of minimum degree in $G$. Then every vertex outside $v\cup N(v)$ has at least $r-1$ edges to $N(v)$ (which contains a $K_{r-2}$). This gives 
\[ |E(G)|\geq (r-1)(n-r+1) +\binom{r-2}{2} . \]
\end{proof}

Note that~\ref{thmkr} implies that there are non-bipartite graphs for which the rainbow saturation number is not asymptotically equal to the traditional saturation number. This is not the case for the rainbow extremal number versus the traditional extremal number.

\section{Bipartite Graphs}
In our proof of Theorem~\ref{thegoods}, we will use a bound on the rainbow saturation numbers for bipartite graphs in an essential way. Unfortunately, this seems quite difficult in general; nevertheless, we can easily obtain linear bounds for the rainbow saturation numbers of trees, and of $C_4$'s.

The result for trees follows immediately from the following bound on rainbow extremal numbers for forests, due to Johnston, Palmer, and Rombach~\cite{JPR}. Since for any $n$ and $H$ we have $\satr(n,H)\le\exr(n,H)$, Proposition~\ref{treeprop} follows.

\begin{lem}\label{treepropex}[\cite{JPR}]
For every forest $F$, there is a constant $c=c(F)$ such that $\exr(n,F)\le cn$.
\end{lem}

\begin{prop}\label{treeprop} For every forest $F$, there is a constant $c=c(F)$ such that  $\satr(n,F)\le cn$.
\end{prop}

For some classes of trees, we can obtain fairly good bounds on constant in Proposition~\ref{treeprop}. Next, we give a linear bound on the rainbow saturation number for paths of length three\footnote{Note that we use $P_k$ to denote the path with $k$ vertices and $(k-1)$ edges. Further, note that since every proper coloring of $P_3$ is rainbow, we know that $\satr(n,P_3)=\sat(n,P_3)$.} and cycles of length four. First, we give the corresponding saturation number for comparison.

\begin{thm}[\cite{KT}]
$\sat(n,P_4)=\begin{cases} \frac{n}{2}&\textrm{ if }n\textrm{ is even,}\\&\\ \frac{(n+3)}{2}&\textrm{ if }n\textrm{ is odd.}\end{cases}$
\end{thm}

\begin{thm}[\cite{FFL, Ollman, TuzaC4,Tuza}]
$\sat(n,C_4)=\left\lfloor\frac{3n-5}{2}\right\rfloor$.
\end{thm}

\begin{thm}\label{p4s}
For each $n\geq 16$ we have 
\[ \satr(n,P_4)\leq\frac{4}{5}n -\frac{17}{10}c  ,\]
where $0\leq c \leq 4$ and $c\equiv-n~\pmod{5}$
\end{thm}

\begin{proof}
First, we provide an upper bound by constructing a rainbow $P_4$-saturated graph. For $n\geq 16$, let $G$ be a graph on $n$ vertices composed of a disjoint union of copies of $K_{1,4}$ and copies of $K_4$, with as many copies of $K_{1,4}$ as possible. This is achieved by letting $c$ be such that $0\leq c \leq 4$ and $c\equiv-n~\pmod{5}$, and taking $c$ copies of $K_4$, and $n-4c$ copies of $K_{1,4}$. Edge-color the copies of $K_4$ using 3 colors. This can only be done in a way that prevents a rainbow copy of $P_4$. The copies of $K_{1,4}$ may have any edge-coloring. Now, $G$ does not contain a rainbow copy of $P_4$. Suppose that we add an edge $e$ to $G$. Note that if one of the endpoints of $e$ has degree at least 2 in $G$, we find a rainbow copy of $P_4$. Otherwise, we create one of the two subgraphs shown in Figure~\ref{figP4subs}. In both cases, it is easy to see that we must have a rainbow copy of $P_4$. This gives
\[ \satr(n,P_4)\leq 
\frac{4(n-4c)}{5} +\frac{6c}{4}
=\frac{4}{5}n -\frac{17}{10}c .\]
 
\begin{figure}[ht]
\begin{center}
\begin{tikzpicture}[scale=1]]
\draw[fill=black!100,black!100] (-1,0) circle (.15);
\draw[fill=black!100,black!100] (1,.5) circle (.15);
\draw[fill=black!100,black!100] (1,-.5) circle (.15);
\draw[fill=black!100,black!100] (1,1.5) circle (.15);
\draw[fill=black!100,black!100] (1,-1.5) circle (.15);
\draw[black!100,line width=1.5pt] (-1,0) -- (1,.5);
\draw[black!100,line width=1.5pt] (-1,0) -- (1,-.5);
\draw[black!100,line width=1.5pt] (-1,0) -- (1,1.5);
\draw[black!100,line width=1.5pt] (-1,0) -- (1,-1.5);
\draw[black!100,line width=1.5pt,dashed] (1,.5) -- (1,1.5);
\draw (1.3,1) node{$e$};
\begin{scope}[xshift=4cm]
\draw[fill=black!100,black!100] (-1,0) circle (.15);
\draw[fill=black!100,black!100] (1,.5) circle (.15);
\draw[fill=black!100,black!100] (1,-.5) circle (.15);
\draw[fill=black!100,black!100] (1,1.5) circle (.15);
\draw[fill=black!100,black!100] (1,-1.5) circle (.15);
\draw[black!100,line width=1.5pt] (-1,0) -- (1,.5);
\draw[black!100,line width=1.5pt] (-1,0) -- (1,-.5);
\draw[black!100,line width=1.5pt] (-1,0) -- (1,1.5);
\draw[black!100,line width=1.5pt] (-1,0) -- (1,-1.5);
\draw[fill=black!100,black!100] (2,1.5) circle (.15);
\draw[black!100,line width=1.5pt,dashed] (2,1.5) -- (1,1.5);
\draw (1.5,1.3) node{$e$};
\end{scope}
\end{tikzpicture}
\caption{Two possible subgraphs in $G+e$. Both must have a rainbow copy of $P_4$ under any edge-coloring.}\label{figP4subs}
\end{center}
\end{figure}
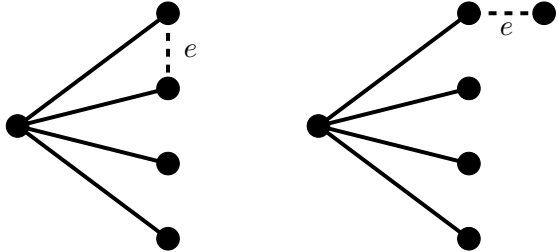

Now, for the lower bound, suppose that $G$ has at least two components on at most 3 edges. If these components have a cycle, they do not decrease the number of edges with respect to $n$, For any acyclic component on 2 or 3 edges, we can add an edge $e$ without creating a rainbow-copy of $P_4$. This is shown in Figure~\ref{figP4subs1}. 

\begin{figure}[h]
\begin{center}
\begin{tikzpicture}[scale=1]]
\draw[fill=black!100,black!100] (0,0) circle (.15);
\draw[fill=black!100,black!100] (1,.5) circle (.15);
\draw[fill=black!100,black!100] (1,-.5) circle (.15);
\draw[black!100,line width=1.5pt] (0,0) -- (1,.5);
\draw[black!100,line width=1.5pt] (0,0) -- (1,-.5);
\draw[black!100,line width=1.5pt,dashed] (1,.5) -- (1,-.5);
\draw (1.3,0) node{$e$};
\begin{scope}[xshift=4cm]
\draw[fill=black!100,black!100] (-1,0) circle (.15);
\draw[fill=black!100,black!100] (1,1) circle (.15);
\draw[fill=black!100,black!100] (1,-1) circle (.15);
\draw[fill=black!100,black!100] (1,0) circle (.15);
\draw[black!100,line width=1.5pt] (-1,0) -- (1,1);
\draw[black!100,line width=1.5pt] (-1,0) -- (1,-1);
\draw[black!100,line width=1.5pt] (-1,0) -- (1,0);
\draw[black!100,line width=1.5pt,dashed] (1,0) -- (1,1);
\draw (1.3,.5) node{$e$};
\end{scope}
\begin{scope}[xshift=7cm]
\draw[fill=black!100,black!100] (0,-.5) circle (.15);
\draw[fill=black!100,black!100] (1,.5) circle (.15);
\draw[fill=black!100,black!100] (0,.5) circle (.15);
\draw[fill=black!100,black!100] (1,-.5) circle (.15);
\draw[black!100,line width=1.5pt] (0,-.5) -- (0,.5);
\draw[black!100,line width=1.5pt] (0,-.5) -- (1,-.5);
\draw[black!100,line width=1.5pt] (0,1-.5) -- (1,.5);
\draw[black!100,line width=1.5pt,dashed] (1,-.5) -- (1,.5);
\draw (1.3,0) node{$e$};
\end{scope}
\end{tikzpicture}
\caption{The only three possible acyclic components on 2 or 3 edges in $G$, each with a new edge $e$ added without forcing a rainbow $P_4$.}\label{figP4subs1}
\end{center}
\end{figure}
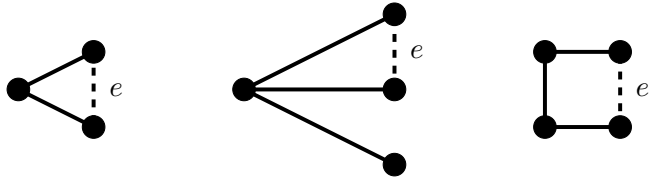
Therefore, the only components of $G$ on at most $3$ edges are either isolated vertices or copies of $K_2$. Adding an edge across two such components creates a component isomorphic to $P_2$, $P_3$ or $P_4$, each of which can be colored without a (rainbow) copy of $P_4$. 

We conclude that $G$ must have components on at least 4 edges. The lowest edge density is then achieved by having acyclic components, as demonstrated in our earlier construction. This gives
\[ \satr(n,P_4)\geq 
\left\lfloor \frac{4n}{5}  \right\rfloor.\]
\end{proof}

For cycles of length four, we can show that the traditional and rainbow saturation numbers are again different.  Though it leaves a wide range of possible values for $\satr(n,C_4)$, even showing this separation is non-trivial.

\begin{thm}\label{c4s}
For $n \geq 4$ we have $$n  \le\satr(n,C_4)\le 2n-2.$$
\end{thm}

\begin{proof}
We start with the upper bound, by providing a construction for a rainbow $C_4$-saturated graph. Let $n \geq 6$, and consider the wheel on $n$ vertices, denoted $W_n$. This is the graph formed by adding a universal vertex to a cycle graph $C_{n-1}$. We label the vertices around the cycle as $v_1,\dots,v_{n-1}$, and the universal vertex as $w$. We add the following edge coloring $c: E(W_n) \to \{ 1,\dots,n-1\}$. Let $c(wv_i)=c(v_{i+1}v_{i+2})=i$ for $1 \leq i \leq n-1$ with vertices labeled cyclically. We have illustrated this graph and edge-coloring in Figure~\ref{figwheel}. Note that every copy of $C_4$ in this graph is of the form $(v_i,v_{i+1},v_{i+2},w)$, such that edges $wv_i$ and $v_{i+1}v_{i+2}$ have the same color.

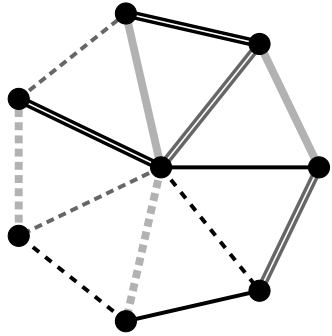
\begin{figure}[h]
\begin{center}
\begin{tikzpicture}[yscale=-0.7,xscale=0.7]
\draw[black!100,line width=1.5pt] (0,0) -- (0:3);
\draw[black!100,line width=1.5pt] (360/7:3) -- (360/7*2:3);
\draw[black!100,line width=1.5pt,dashed] (0,0) -- (360/7:3);
\draw[black!100,line width=1.5pt,dashed] (360/7*2:3) -- (360/7*3:3);
\draw[black!30,line width=3pt,dotted] (0,0) -- (360/7*2:3);
\draw[black!30,line width=3pt,dotted] (360/7*3:3) -- (360/7*4:3);
\draw[black!30,line width=3pt] (0,0) -- (360/7*5:3);
\draw[black!30,line width=3pt] (360/7*6:3) -- (0:3);
\draw[black!60,line width=1.3pt,double] (0,0) -- (360/7*6:3);
\draw[black!60,line width=1.3pt,double] (0:3) -- (360/7:3);
\draw[black!60,line width=1.5pt,densely dashed] (0,0) -- (360/7*3:3);
\draw[black!60,line width=1.5pt,densely dashed] (360/7*4:3) -- (360/7*5:3);
\draw[black!100,line width=1.3pt, double] (0,0) -- (360/7*4:3);
\draw[black!100,line width=1.3pt, double] (360/7*5:3) -- (360/7*6:3);
\draw[fill=black!100,black!100] (0,0) circle (.2);
\draw[fill=black!100,black!100] (360/7:3) circle (.2);
\draw[fill=black!100,black!100] (360/7*2:3) circle (.2);
\draw[fill=black!100,black!100] (360/7*3:3) circle (.2);
\draw[fill=black!100,black!100] (360/7*4:3) circle (.2);
\draw[fill=black!100,black!100] (360/7*5:3) circle (.2);
\draw[fill=black!100,black!100] (360/7*6:3) circle (.2);
\draw[fill=black!100,black!100] (0:3) circle (.2);
\end{tikzpicture}
\caption{A wheel graph $W_8$ with an edge-coloring that avoids a rainbow copy of $C_4$.}\label{figwheel}
\end{center}
\end{figure}

Now, suppose that we add a new edge $e$ to the graph $W_n$. This gives rise to one of two new subgraphs, depending on whether or not $e$ is of the form $v_{i}v_{i+2}$. Both subgraphs are shown in Figure~\ref{figwheelsubs}. Note that we relabeled the vertices for ease of notation.

\begin{figure}[h]
\begin{center}
\begin{tikzpicture}[scale=1]]
\draw[fill=black!100,black!100] (0,0) circle (.15);
\draw[fill=black!100,black!100] (30:1.5) circle (.15);
\draw[fill=black!100,black!100] (60:1.5) circle (.15);
\draw[fill=black!100,black!100] (90:1.5) circle (.15);
\draw[fill=black!100,black!100] (120:1.5) circle (.15);
\draw[fill=black!100,black!100] (150:1.5) circle (.15);
\draw[black!100,line width=1.5pt] (0,0) -- (30:1.5);
\draw[black!100,line width=1.5pt] (0,0) -- (60:1.5);
\draw[black!100,line width=1.5pt] (0,0) -- (90:1.5);
\draw[black!100,line width=1.5pt] (0,0) -- (120:1.5);
\draw[black!100,line width=1.5pt] (0,0) -- (150:1.5);
\draw[black!60,thick,dashed] (0,0) circle (1.5cm);
\draw[black!100,line width=1.5pt] (30:1.5) to[out=120,in=-30] (60:1.5);
\draw[black!100,line width=1.5pt] (60:1.5) to[out=150,in=0] (90:1.5);
\draw[black!100,line width=1.5pt] (90:1.5) to[out=180,in=30] (120:1.5);
\draw[black!100,line width=1.5pt] (120:1.5) to[out=210,in=60] (150:1.5);
\draw[black!100,line width=1.5pt] (120:1.5) to[out=180,in=-90]  (-2,2);
\draw[black!100,line width=1.5pt] (60:1.5) to[out=0,in=-90]  (2,2);
\draw[black!100,line width=1.5pt] (-2,2) to[out=90,in=90]  (2,2);
\draw (0,-.3) node{$w$};
\draw (-1.7,0.5) node{$a$};
\draw (1.7,0.5) node{$e$};
\draw (-1,1.5) node{$b$};
\draw (1,1.5) node{$d$};
\draw (0,1.9) node{$c$};
\draw (0,3.5) node{$G_A$};
\begin{scope}[xshift=6cm]
\draw[fill=black!100,black!100] (0,0) circle (.15);
\draw[fill=black!100,black!100] (-30:1.5) circle (.15);
\draw[fill=black!100,black!100] (0:1.5) circle (.15);
\draw[fill=black!100,black!100] (30:1.5) circle (.15);
\draw[fill=black!100,black!100] (150:1.5) circle (.15);
\draw[fill=black!100,black!100] (180:1.5) circle (.15);
\draw[fill=black!100,black!100] (210:1.5) circle (.15);
\draw[black!100,line width=1.5pt] (0,0) -- (-30:1.5);
\draw[black!100,line width=1.5pt] (0,0) -- (0:1.5);
\draw[black!100,line width=1.5pt] (0,0) -- (30:1.5);
\draw[black!100,line width=1.5pt] (0,0) -- (150:1.5);
\draw[black!100,line width=1.5pt] (0,0) -- (180:1.5);
\draw[black!100,line width=1.5pt] (0,0) -- (210:1.5);
\draw[black!60,thick,dashed] (0,0) circle (1.5cm);
\draw[black!100,line width=1.5pt] (-1.5,0) to[out=180,in=-90]  (-2.5,1);
\draw[black!100,line width=1.5pt] (1.5,0) to[out=0,in=-90]  (2.5,1);
\draw[black!100,line width=1.5pt] (-2.5,1) to[out=90,in=90]  (2.5,1);
\draw[black!100,line width=1.5pt] (-30:1.5) to[out=60,in=-90] (0:1.5);
\draw[black!100,line width=1.5pt] (0:1.5) to[out=90,in=-60] (30:1.5);
\draw[black!100,line width=1.5pt] (150:1.5) to[out=245,in=90] (180:1.5);
\draw[black!100,line width=1.5pt] (180:1.5) to[out=-90,in=120] (210:1.5);
\draw (0,-.3) node{$w$};
\draw (-1.3,-1.1) node{$a$};
\draw (-1.7,-.3) node{$b$};
\draw (1.7,-.3) node{$e$};
\draw (-1.3,1.1) node{$c$};
\draw (1.3,1.1) node{$d$};
\draw (1.3,-1.1) node{$f$};
\draw (0,2.9) node{$G_B$};
\end{scope}
\end{tikzpicture}
\caption{Two possible subgraphs in $W_n+e$. Neither of these graphs have a rainbow $C_4$-free edge-coloring.}\label{figwheelsubs}
\end{center}
\end{figure}
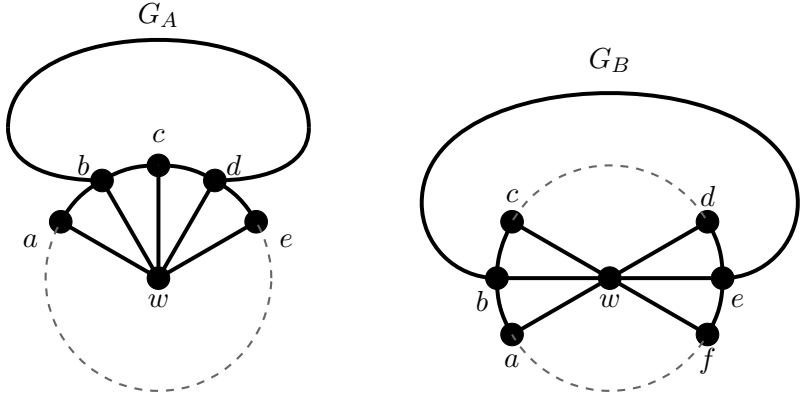
We will now show that neither of the subgraphs in Figure~\ref{figwheelsubs} has an edge-coloring with no rainbow copy of $C_4$.

Consider the subgraph $G_A$. The edges $bw$, $wd$ and $bd$ form a triangle and must have distinct colors. Label these colors $1$, $2$ and $3$, respectively. Consider the following 4 copies of $C_4$ (with slightly relaxed notation): $bcwd$, $bwcd$, $bawd$ and $bwed$. Suppose that none of them are rainbow colored. Then:
\begin{itemize}
    \item $c(bc)=2$ or $c(cw)=3$, and
    \item $c(cd)=1$ or $c(cw)=3$, and
    \item $c(ba)=2$ or $c(aw)=3$, and
    \item $c(ed)=1$ or $c(we)=3$.
\end{itemize}
It is possible to fulfill all these 4 conditions at once. This can be done in only one way, which is by setting $c(cw)=3$, $c(ba)=2$ and $c(ed)=1$. However, we now have that $c(bc)\neq 2$ and $c(cd)\neq 1$. This implies that the cycle $bwdc$ is rainbow colored. 

Consider the subgraph $G_B$. The edges $bw$, $we$ and $be$ form a triangle and must have distinct colors. Label these colors $1$, $2$ and $3$, respectively. Consider the following 4 copies of $C_4$: $bcwe$, $bwde$, $bawe$ and $bwfe$. Suppose that none of them are rainbow colored. Then:
\begin{itemize}
    \item $c(bc)=2$ or $c(cw)=3$, and
    \item $c(de)=1$ or $c(dw)=3$, and
    \item $c(ba)=2$ or $c(aw)=3$, and
    \item $c(fe)=1$ or $c(wf)=3$.
\end{itemize}
In this case, these 4 conditions cannot be satisfied simultaneously, and we must therefore have a rainbow copy of $C_4$ within this set. This gives the upper bound
\[ \satr(n,C_4)\le 2n-2 .\]

For the lower bound, first note that any rainbow $C_4$-saturated graph must be connected, as adding a bridge to $G$ can never force a new rainbow copy of $C_4$. Suppose that $G$ is rainbow $C_4$-saturated and has two non-adjacent vertices $v,w$ of degree 1, with neighbors $v'$ and $w'$, respectively. Consider adding the edge $vw$ to $G$. Since $G$ has an edge-coloring that is rainbow $C_4$-free, the only place where we could find a rainbow $C_4$ is on the cycle $(v,v',w',w)$. If $v'=w'$ this is not a $C_4$. Otherwise, let $c(vw)=c(v'w')$. Now, $G+vw$ does not have a rainbow copy of $C_4$, and therefore $G$ is not rainbow $C_4$-saturated. We conclude that, for $n\geq 4$, we must have that $G$ has at most one vertex of degree 1. Since $G$ is connected every other vertex has degree at least two; if it has a single vertex of degree one, then it also has a vertex of degree at least three, and via the handshake lemma $||G||\ge\left\lceil \frac{2(n-2)+3+1}{2}\right\rceil=n$. If $G$ has no degree one vertex, then $\|G\|\ge\frac{2n}{2}=n$. We conclude that
\[ \satr(n,C_4)\ge n . \]
\end{proof}

\section{General Results}
Among the best known results in extremal graph theory is a result of Erd\H{o}s and Stone~\cite{ES}: for any graph $H$ we have $$\ex(n,H)=\left(1-\frac{1}{\chi(G)-1}+o(1)\right)\binom{n}{2}.$$  As a consequence, when the forbidden graph is at least 3-chromatic, the extremal graphs are thus dense. However, this phenomenon fails to appear for saturated graphs in a very strong way -- all saturation numbers are actually linear.

\begin{thm}[K\'aszonyi-Tuza~\cite{KT}] For every graph $H$ there is a constant depending only on $H$ such that $\sat(n,H)\le cn$ for n sufficiently large.\end{thm}

One's intuition may lead them to believe that this is not true for rainbow saturation -- surely with the freedom we get from colorings we can avoid at least some rainbow graphs using a quadratic number of edges!  Nevertheless, we present here an identical statement for the rainbow saturation numbers of a large class of graphs.

\begin{thm}\label{thegoods} Suppose that $H$ is even-cycle-free\footnote{That is, $H$ contains no induced even cycle}. Then, there is a constant $c$, depending only on $H$, such that $\satr(n,H)\le cn$.\end{thm}

\begin{proof}
Let $H$ be a graph that is even-cycle-free. We let $\F^{(0)}=\{ H\}$. Now, we define a family of graph sets $\F^{(i)}|_{i=0}^k$ as follows. If the set $\F^{(j)}$ contains a bipartite graph, then we set $k=j$. Otherwise, we let $\F^{(j+1)}$ be the set of all graphs which can be obtained from a graph in $\F^{(j)}$ by removing a maximal independent set. Letting $\mathcal{I}(F)$ denote the set of maximal independent sets in $F$, this family is 
$$\F^{(i+1)}=\left\{F-X:F\in \F^{(i)}, \;\mbox{ and } X\in\mathcal{I}(F)\right\}.$$

We describe a recursive construction that results in a rainbow $H$-saturated graph on $n$ vertices, along the lines of~\cite{KT}, but with a few adaptations that will ensure that adding an edge forces a rainbow copy of $H$. We will start by finding a graph $G^{(k)}$ that is rainbow $\F^{(k)}$-saturated, then a $G^{(k-1)} \supset G^{(k)} $ that is rainbow $\F^{(k-1)}$-saturated, etc, until we have
$$G^{(0)}\supset G^{(1)}\supset \dots \supset G^{(k-1)} \supset G^{(k)},$$
where $G^{(0)}$ is a rainbow $H$-saturated graph on $n$ vertices.

We note that not each graph in $\F^{(i)}$ has the same order. For reasons that will become apparent later, we let $h=|H|$ and write $n_i$ in a slightly strange (but later useful) form as $$n_i=n-i (h^3+h)=n+O(1),\;\; 1\leq i\leq k.$$

By construction of the family $\F^{(i)}|_{i=0}^k$, we have that $\F^{(k)}$ contains a bipartite graph. Since $H$ is even-cycle-free, such a bipartite graph $B \in \F^{(k)}$ must be a forest. We construct a rainbow-$\F^{(k)}$-saturated graph greedily as follows. Start with an empty graph on $n_k$ vertices. Clearly this graph is $\F^{(k)}$-free. While it is not rainbow-$\F^{(k)}$-saturated, we sequentially add edges that do not force any rainbow-copies of graphs in $\F^{(k)}$ until we arrive at a rainbow-$\F^{(k)}$-saturated graph\footnote{We note that adding edges in a different order may result in a different graph; any such graph will suffice.}. We call this graph $G^{(k)}$. Since $\F^{(k)}$ contains a forest $B$, Proposition~\ref{treepropex} guarantees that $G^{(k)}$ has at most $c'n_k$ edges, where $c'$ is a constant that depends only on $B$.

Now, assume that we have a $G^{(i)}$, for $1\leq i \leq k$, which is a rainbow $\F^{(i)}$-saturated graph on $n_i$ vertices. In the next step, we will add $h^3+h$ vertices to this graph, in addition to a set of edges that are all incident to at least one of the new vertices, to create a $G^{(i-1)}$. We will show that this $G^{(i-1)}$ is a rainbow $\F^{(i-1)}$-saturated graph on $n_{i-1}$ vertices and on $\| G^{(i)} \|+O(n)$ edges. Since we start with $O(n)$ edges, and add $O(n)$ edges in each of the $k$ steps, our final graph $G^{(0)}$ will have $O(n)$ edges. It remains to be shown that we can build such a $G^{(i-1)}$, given a $G^{(i)}$.

First we let $I$ be an independent set consisting of $h^3+h$ vertices, which we join to $G^{(i)}$. We call the resulting graph $G^{(i)*}$. We note that, since $G^{(i)}$ is rainbow $\F^{(i)}$-saturated, $G^{(i)*}$ must be rainbow $\F^{(i-1)}$-free. Further, since each of these added vertices in $I$ have the same neighborhood, by making $I$ large enough, we can ensure that whenever an edge is added to $G^{(i)}$, we find a rainbow graph from $\F^{(i-1)}$ in the newly constructed graph $G^{(i)*}$. This rainbow copy is formed by $F+X$, where $F \in \F^{(i)}$ and $X \subset I$. 

How large does $I$ need to be, in order to ensure that such an $X$ exists? Suppose that we have added an edge $e$ to $G^{(i)}$, given $G^{(i)*}$ an arbitrary edge-coloring, and found a rainbow copy $F$ of a graph in $\F^{(i)}$ inside $G^{(i)}+e$. We now need to find a set $X \subseteq I$ with $X$ contained in a maximal independent set such that $X$ can be added to $F$ to obtain a rainbow copy of a graph $F+X \in \F^{(i-1)}$. In the worst case, we need all edges between $X$ and $V(F)$. We can find a suitable $X$ by starting with $X=\emptyset$ and adding vertices to $X$ one by one. A vertex can be added to $X$ if all of its edges to $V(F)$ have colors other than colors already used in our rainbow copy $F$ and colors of edges from current vertices in $X$ to $V(F)$. Therefore, at any time, there are no more than $h^3$ vertices in $I \setminus X$ that cannot be added to $X$. Therefore, if we let $|I|=h^3+h$, such a set $X \subset I$ is guaranteed to exist. For our purposes, we mainly need that $|I|=O(1)$, but we need to specify a cardinality for $I$ in order to define an explicit construction of $G^{(i)}|_{i=0}^k$.

This graph $G^{(i)*}$ still may not be rainbow $\F^{(i-1)}$-saturated; thus far, we have only shown that adding edges to $G^{(i)} \subset G^{(i)*} $ will force a rainbow $\F^{(i-1)}$-graph; by construction, this rainbow $\F^{(i-1)}$-graph indeed forces a rainbow copy of $H$. What about adding an edge between two vertices in $I$? This may not force a rainbow $\F^{(i-1)}$-graph. If this is the case for any pair of vertices in $I$, we simply add this edge to $G^{(i)*}$. Repeating this for any problematic pairs in $I$, we obtain a rainbow $\F^{(i-1)}$-saturated graph, and we let this graph be $G^{(i-1)}$. Since $|I|=h^3+h\leq h^4$ and $\| G^{(k)}  \| \leq cn_k$, we have  
\[\| G^{(0)} \|  \leq c'n_k + kh^4(n_k+n_{k-1}+\dots + n_1) + k^2h^8 \leq cn, \]
for some constant $c$ that depends only on the graph $H$.
\end{proof}

\section{Further Questions}
One may take virtually any question in the long and involved history of saturation and extremal numbers, add the word `rainbow' to it, and obtain an interesting (and likely as-of-yet-unsolved) question. In this section, we collect a few problems that the authors find particularly interesting.

Naturally, we are very interested in extending Theorem~\ref{thegoods} to all graphs. We state this as a conjecture here.

\begin{conj}
For every graph $H$ there is a constant $c=c(H)$ such that $\satr(n,H)\leq cn$.
\end{conj}

We would also like to find a matching lower bound for Theorem~\ref{c4s}, and extend this to other cycle lengths. We suspect that even and odd cycles will behave quite differently.

\begin{problem}
Determine $\satr(n,C_k)$ for each $k\ge 3$.
\end{problem}

Recall that $\satr(n,H)$ and $\exr(n,H)$ numbers are the minimum and maximum number of edges in a rainbow $H$-saturated graph. Which other numbers of edges are possible to achieve? This is the rainbow version of determining the {\emph{saturation spectrum}} of $H$; similar problems were first explored in~\cite{barefoot} and extended in a long series of papers by assorted authors; see~\cite{RG} for details.

\begin{question}
For $m\in[\satr(n,H),\exr(n,H)]$, is there an $n$-vertex rainbow $H$-saturated graph with exactly $m$ edges?
\end{question}

In the extremal number milieu, one can easily deduce that $\exr(n,H)\ge\ex(n,H)$ since every $H$-free graph is trivially rainbow $H$-free. However, a graph which is $H$-saturated may not be rainbow $H$-saturated (since there could be a proper edge coloring which avoids rainbow copies of $H$, even if the underlying graph has some copies of $H$).

\begin{question}
Is there an example of some graph $H$ for which $\satr(n,H)<\sat(n,H)$?
\end{question}

Alon and Shikhelman defined the generalized Tur\'an number $ex(n,F,H)$ to be the largest number of copies of $F$ in an $H$-free graph. This was extended recently to the rainbow setting by Gerbner, M\'esz\'aros, Methuku, and Palmer~\cite{GMMP}. We define here the corresponding generalized rainbow saturation number.

\begin{definition}
For graphs $F$ and $H$, we define $\satr(n,F,H)$ to be the minimum number of copies of $F$ in any rainbow $H$-saturated graph.
\end{definition}

The results in this paper correspond to determining $\satr(n,K_2,H)$. It would be very interesting to replace $K_2$ with any other graph. Of course, if $F\supseteq H$, then this is trivially $0$.

\begin{problem}
Determine $\satr(n,F,H)$ for any non-trivial graphs $F$ and $H$.
\end{problem}

Finally, we end with a conjecture about rainbow saturation in hypergraphs. Pikhurko~\cite{Pikhurko} proved an extension of the K\'aszonyi and Tuza result mentioned previously, showing that $\sat_k (n,H)=O(n^{k-1})$ for any $k$-uniform hypergraph $H$. We conjecture that this holds for rainbow saturation, as well. The method of proof from Theorem~\ref{thegoods} works here with minor modification; however, we have no hypergraph equivalent of Theorem~\ref{treeprop} to end the independent set removal process.

\begin{conj}
For every $k$-uniform hypergraph $H$, we have $\satr_k(n,H)=O(n^{k-1})$.
\end{conj}


\begin{thebibliography}{99}

\bibitem{AS} N. Alon, C. Shikhelman. \emph{Many $T$ copies in $H$-free graphs}. Journal of Combinatorial Theory, Series B {\bf{121}}, pp 146--172 (2016).

\bibitem{barefoot} C. Barefoot, K. Casey, D. Fisher, K. Fraughnaugh, F. Harary. \emph{Size in
maximal triangle-free graphs and minimal graphs of diameter 2}. Discrete
Mathematics {\bf{138}}(1-3), pp. 93--99 (1995).

\bibitem{rain1} M. Barrus, M. Ferrara, J. Vandenbussche, P. Wenger. \emph{Colored saturation
parameters for rainbow subgraphs}. Journal of Graph Theory {\bf{86}}, pp. 375--386 (2017).

\bibitem{MGT} B.~Bollob\'as. \emph{Modern Graph Theory}. Springer-Verlag, New York (2002).


\bibitem{CFFS} B.~L.~Currie, J.~R.~Faudree, R.~J.~Faudree, J.~R.~Schmitt. \emph{A survey of minimum maturated graphs}. Electronic Journal of Combinatorics DS19 (2021).

\bibitem{EHM} P.~Erd\H{o}s, A.~Hajnal, J.~Moon. \emph{A problem in graph theory}. American Mathematical Monthly {\bf{71}} pp. 1107--1110 (1964).

\bibitem{ES} P.~Erd\H{o}s, A.~Stone. \emph{On the structure of linear graphs}. Bulletin of the American Mathematical Society {\bf{52}} pp. 1087--1091 (1946).

\bibitem{FFL} D. Fisher, K. Fraughnaugh, L. Langley. \emph{$P_3$-connected graphs of minimum size}. Ars Combinatorica {\bf{47}}, pp. 299--306 (1997).

\bibitem{GMMP} D. Gerbner, T. M\'esz\'aros, Abhishek Methuku, Cory Palmer. \emph{Generalized rainbow Tur\'an Problems}. arXiv:1911.06642v1 (2019).

\bibitem{RG} R. Gould. \emph{Developments on Saturated Graphs}, in F. Chung,  R. Graham, L. Hogben, F. Hoffman, R.Mullin, D. West (eds.). \emph{50 Years of Combinatorics, Graph Theory, and Computing}, CRC Press (2019).

\bibitem{MantelSol} H. Gouwentak, W. Mantel, J.Teixeira de Mattes, F. Schuh, W. A. Wythoff. {\emph{Problem 28}}. Wiskundige Opgaven met de Oplossingen {\bf{10}}, pp. 60--61 (1910).

\bibitem{JPR} D.~Johnston, C.~Palmer, P.~Rombach. \emph{Generalized Rainbow Tur\'an Numbers}. Manuscript.


\bibitem{KT} L.~K\'aszonyi, Z.~Tuza. \emph{Saturated graphs with minimal number of edges}.
Journal of Graph Theory {\bf{10}}, pp. 203--210 (1986).

\bibitem{KMSV} P. Keevash, D. Mubayi, B. Sudakov, and J. Verstra\"ete. \emph{Rainbow Tur\'an problems}. Combinatorics, Probability and Computing, 16(1), pp. 109-–126 (2007).

\bibitem{rain2} D. Kor\'andi. \emph{Rainbow saturation and graph capacities}. SIAM Journal of Discrete
Mathematics {\bf{32}}(2), pp. 1261--1264 (2018).

\bibitem{Ollman} L. Ollman. \emph{$K_{2,2}$-saturated graphs with a minimal number of edges}. Proceedings of the 3rd Southeastern Conference on Combinatorics, Graph Theory, and Computing, pp. 367–392 (1972).

\bibitem{Pikhurko} O. Pikhurko. \emph{Results  and  open  problems  on  minimum  saturated  hypergraphs}. Ars Combinatorica, {\bf{72}}, pp. 111--127 (2004).

\bibitem{Turan} P. Tur{\'a}n. \emph{On an Extremal Problem in Graph Theory}. Mat. Fiz. Lap. (in Hungarian) {\bf{48}}, pp. 436--452 (1941).

\bibitem{TuzaC4} Z. Tuza. \emph{$C_4$-saturated graphs of minimum size}. Acta Univ. Carolin. Math. Phys. {\bf{30-2}}, pp. 161--167 (1989).

\bibitem{Tuza} Z. Tuza. \emph{Extremal problems on saturated graphs and hypergraphs}. Ars Combinatorica, Eleventh British Combinatorial Conference (London 1987) {\bf{25b}}, pp. 105--113 (1988).



\end{thebibliography}
\end{document}